\documentclass[a4paper]{amsart}                

\usepackage[latin1]{inputenc}                  
\usepackage[T1]{fontenc}                       
\usepackage[spanish,english]{babel}            
\usepackage[pdftex]{graphicx}                  
\usepackage{amsmath,amssymb,amsthm}            
\usepackage{latexsym}                          
\usepackage{delarray}                          
\usepackage{bbm}                               
\usepackage{hyperref}                          
\usepackage[pdftex,usenames,dvipsnames]{color} 
\usepackage{bbding,trfsigns}                   
\usepackage{wasysym}                           
\usepackage{datetime}                          

\DeclareMathAlphabet{\mathpzc}{OT1}{pzc}{m}{it} 

\newtheorem{Th}{Theorem}[section]              

\newtheorem{Prop}{Proposition}[section]
\newtheorem{Lem}{Lemma}[section]

\newcommand{\B}{\mathbb{B}}

\title[UMD spaces and imaginary powers of Hermite and Laguerre operators]
{Characterization of UMD Banach spaces by imaginary powers of Hermite and Laguerre operators}

\author[J. Betancor]{Jorge J. Betancor}
\address{Jorge J. Betancor, Alejandro J. Castro, and Lourdes Rodríguez-Mesa \newline
Departamento de An\'{a}lisis Matem\'{a}tico\\
Universidad de la Laguna\\
Campus de Anchieta, Avda. Astrof\'{\i}sico Francisco S\'{a}nchez, s/n\\
38271 La Laguna (Sta. Cruz de Tenerife), Spain}
\email{jbetanco@ull.es, ajcastro@ull.es, lrguez@ull.es}

\author[A.J. Castro]{Alejandro J. Castro}

\author[J. Curbelo]{Jezabel Curbelo}
\address{Jezabel Curbelo\newline
Instituto de Ciencias Matem\'aticas (CSIC-UAM-UC3M-UCM)\\
Consejo Superior de Investigaciones Cient\'{\i}ficas\\
Nicol\'as Cabrera 15\\
28049 Madrid, Spain} \email{jezabel.curbelo@icmat.es}

\author[L. Rodr\'{\i}guez-Mesa]{Lourdes Rodr\'{\i}guez-Mesa}

\thanks{This paper is partially supported by MTM2010/17974. The second author is also supported by a FPU grant from the Government of
Spain and the third one by a grant JAE-Predoc of the CSIC}

\begin{document}

\maketitle

\begin{abstract}
In this paper we characterize the Banach spaces with the UMD property by means of $L^p$-boundedness properties for the imaginary powers of the Hermite and Laguerre operators. In order to do this we need to obtain pointwise representations for the Laplace transform type multipliers associated with Hermite and Laguerre operators.
\end{abstract}

\section{Introduction}

 A Banach space $\B$ is said UMD when, for some (equivalently, for any) $1<p<\infty$ the $\B$-valued Hilbert transform defined on $L^p(\mathbb{R})\otimes \B$ can be extended as a bounded operator on the Bochner-Lebesgue space $L^p_\B(\mathbb{R})$ (see \cite{Bour} and \cite{Burk}).

Characterizations of the UMD spaces involving $L^p $-boundedness properties for singular integral operators or $g$-functions have been established by several authors (\cite{AT}, \cite{Guer}, \cite{HTV}, \cite{Hyto1}, \cite{Hyto2} and \cite{Xu}). More concretely, UMD Banach spaces are described by means of the $L^p$-boundedness of the imaginary powers of Laplacian in \cite{Guer}. Our objective in this paper is to characterize the UMD Banach spaces as those Banach spaces $\B$ for which the imaginary powers $\mathcal{L}^{i\gamma}$, $\gamma\in \mathbb{R}$, of $\mathcal{L}$ can be extended to $L^p_\B(\Omega,\mu)$ as bounded operators on $L^p_\B(\Omega,\mu)$, $1<p<\infty$, when $\mathcal{L}$ represent the Hermite or Laguerre operators and $(\Omega,\mu)$ the associated measure space.

Suppose that $\Omega\subset \mathbb{R}^n$, $\mu$ is a positive measure on $\Omega$, $\widetilde{\mathcal{L}}$ is a second order linear differential operator defined on $C^2(\Omega)$ and, for every $k\in \mathbb{N}$, $\varphi_k\in L^2(\Omega,\mu)$ is an eigenfunction of $\widetilde{\mathcal{L}}$ associated with $\nu_k\in (0,\infty)$, that is, $\widetilde{\mathcal{L}}\varphi_k=\nu_k\varphi_k$. Assume also that $\{\nu_k\}_{k\in \mathbb{N}}\uparrow \infty$ and that $\{\varphi_k\}_{k\in \mathbb{N}}$ is an orthonormal basis in $L^2(\Omega,\mu)$. We define the operator $\mathcal{L}$ as follows
\begin{equation}\label{A1.1}
\mathcal{L}(f)=\sum_{k=0}^\infty \nu_k c_k(f)\varphi_k,\,\,\,f\in D(\mathcal{L}),
\end{equation}
where, for every $k\in \mathbb{N}$ and $f\in L^2(\Omega,\mu)$, $c_k(f)=\int_\Omega f(x)\varphi_k(x)d\mu(x)$ and
$$
D(\mathcal{L})=\{f\in L^2(\Omega,\mu):\,\sum_{k=0}^\infty \nu_k^2|c_k(f)|^2<\infty\}.
$$
$\mathcal{L}$ is a self-adjoint and positive operator.

If $m:\{\nu_k\}_{k\in \mathbb{N}}\longrightarrow \mathbb{C}$ is bounded we define the spectral multiplier $T_m^\mathcal{L}$  for the operator $\mathcal{L}$ by
$$
T_m^\mathcal{L}(f)=\sum_{k=0}^\infty m(\nu_k)c_k(f)\varphi_k,\,\,\,f\in L^2(\Omega,\mu).
$$
It is clear that $T_m^\mathcal{L}$ is a bounded operator from $L^2(\Omega,\mu)$ into itself.

The semigroup of operators generated by $-\mathcal{L}$ in $L^2(\Omega,\mu)$ is $\{W_t^\mathcal{L}\}_{t>0}$, where, for every $t>0$, $W_t^\mathcal{L}$ is the spectral multiplier defined by
\begin{equation}\label{A2}
W_t^\mathcal{L}(f)=\sum_{k=0}^\infty e^{-t\nu_k}c_k(f)\varphi_k,\,\,\,f\in L^2(\Omega,\mu).
\end{equation}

Following \cite{Stei} we consider the Laplace transform type spectral multipliers associated to the operator $\mathcal{L}$. We say that a continuous function $m$ on $(0,\infty)$ is of Laplace transform type when it is given by
\begin{equation}\label{m}
m(\lambda )=\lambda \int_0^\infty e^{-\lambda t}\phi (t)dt,\quad \lambda \in (0,\infty ),
\end{equation}
where $\phi \in L^\infty (0,\infty )$. Note that $m$ is also a bounded measurable function on $(0,\infty )$. Then, the so called Laplace transform type spectral multiplier $T_m^\mathcal{L}$ is bounded in $L^2(\Omega,\mu)$. In \cite[Corollary 3, p. 121]{Stei} it was established that if $\{W_t^\mathcal{L}\}_{t>0}$ is a symmetric diffusion semigroup (\cite[p. 65]{Stei}) the Laplace transform type spectral multiplier $T_m^\mathcal{L}$ is bounded from $L^p(\Omega,\mu)$ into itself, for every $1<p<\infty$. Many authors have analyzed $L^p$-boundedness properties for Laplace transform type spectral multipliers in different settings (see, for example, \cite{BCC}, \cite{BMaR}, \cite{DDD}, \cite{GCMST}, \cite{GIT}, \cite{Mart}, \cite{Sass}, \cite{Stei}, \cite{Szar} and \cite{Wr}).

A remarkable particular case of Laplace transform type spectral multiplier is the imaginary power $\mathcal{L}^{i\gamma }$, $\gamma \in \mathbb{R}$, defined as usual by $\mathcal{L}^{i\gamma}=T_{m_\gamma }^\mathcal{L}$, where, for every $\gamma \in \mathbb{R}$, $m_\gamma$ represents the function of Laplace transform type given by \eqref{m}, with $\phi(t)=\phi _\gamma (t)=(\Gamma (1-i\gamma ))^{-1}t^{-i\gamma }$, $t\in (0,\infty )$.

In this paper we consider two differential operators:

$\bullet$ The Hermite (harmonic oscillator) operator $\widetilde{H}=-\frac{1}{2}\Big(\frac{d^2}{dx^2} -x^2\Big)$, on $(\mathbb{R},dx)$.

$\bullet$ The Laguerre operator $\widetilde{L}_\alpha=-\frac{1}{2}\left(\frac{d^2}{dx^2}-x^2-\frac{\alpha ^2-1/4}{x^2}\right)$, $\alpha>-1/2$, on $((0,\infty),dx)$.

To simplify the calculations we consider the
Hermite operator on $\mathbb{R}$. The same results can be
established in higher dimensions.

In the sequel we represent by $\mathcal{L}$ (respectively, $(\Omega,\mu)$) one of the operators $H$ or $L_\alpha$, $\alpha>-1/2$,
defined by \eqref{A1.1}(respectively, $(\mathbb{R},dx)$ or $((0,\infty),dx)$). There exists a $C^\infty((0,\infty)\times\Omega\times\Omega)$-function
$$
(t,x,y)\in (0,\infty)\times\Omega\times\Omega \longmapsto W_t^\mathcal{L}(x,y)\in \mathbb{R},
$$
such that, for every $f\in L^2(\Omega,\mu)$,
\begin{equation}\label{rep}
W_t^\mathcal{L}(f)(x)=\int_\Omega W_t^\mathcal{L}(x,y)f(y)d\mu(y),\,\,\,x\in \Omega\,\,and\,\,t>0.
\end{equation}
Moreover, for every $t>0$, $W_t^\mathcal{L}$ can be extended by \eqref{rep} to $L^p(\Omega,\mu)$ as a bounded operator from $L^p(\Omega,\mu)$ into itself, $1\le p\le\infty$.

In order to establish our characterization of the UMD Banach spaces we need to prove pointwise representations for the Laplace transform type spectral multipliers as principal value integral operators.

\begin{Th}\label{Principalvalue}
Let $m$ be a function of Laplace transform type and $\phi \in L^\infty (0,\infty )$ connected with $m$ by \eqref{m}. Then, there exists $\Lambda\in L^\infty(0,\infty)$ such that, for every $f\in L^2(\Omega,\mu)$,
\begin{equation}\label{Tmexpression}
T_m^\mathcal{L} (f)(x)=\lim_{\varepsilon\rightarrow 0^+}\left(\Lambda (\varepsilon )f(x)+\int_{y\in \Omega,\,|x-y|>\varepsilon}K_\phi ^\mathcal{L} (x,y)f(y)d\mu(y)\right),\quad (\mu)-\mbox{a.e. } x\in \Omega,
\end{equation}
where
\begin{equation*}\label{Kalpha}
K_\phi ^\mathcal{L} (x,y)=\int_0^\infty \phi (t)\Big(-\frac{\partial}{\partial t}\Big)W_t^\mathcal{L} (x,y)dt,\quad x,y\in \Omega.
\end{equation*}
Moreover, the limit in \eqref{Tmexpression} exists for every $f\in L^p(\Omega,\mu)$, $1\le p<\infty$, and $T_m^\mathcal{L}$ can be extended by \eqref{Tmexpression} to $L^p(\Omega,\mu)$ as a bounded operator from $L^p(\Omega,\mu)$ into itself, when $1<p<\infty$, and from $L^1(\Omega,\mu)$ into $L^{1,\infty }(\Omega,\mu)$.

If there exists the limit $\lim_{t\to 0^+}\phi(t)=\phi(0^+)$, then
\begin{equation}\label{Tmexpression2}
T_m^\mathcal{L} (f)(x)=\phi(0^+)f(x)+\lim_{\varepsilon \rightarrow 0^+}\int_{y\in \Omega,\,|x-y|>\varepsilon}K_\phi^\mathcal{L} (x,y)f(y)d\mu(y),\quad (\mu)-\mbox{a.e. }x\in \Omega,
\end{equation}
for every $f\in L^p(\Omega, \mu)$, $1\leq p<\infty$.
\end{Th}

The version of Theorem \ref{Principalvalue} for the Ornstein Uhlenbeck operator $\tilde{\mathcal{O}}=-\frac{1}{2}\frac{d^2}{dx^2}+x\frac{d}{dx}$ was
proved by Garc\'{\i}a-Cuerva, Mauceri, Sj\"ogren and Torrea
(\cite[Theorem 4.1]{GCMST}). The value $0$ is an eigenvalue of $\tilde{\mathcal{O}}$ and in this case it is assumed that $m(0)=0$. In the proof of Theorem \ref{Principalvalue} for the Hermite operator (see Subsection 2.1) we will use \cite[Theorem 4.1]{GCMST}.

Our characterization of the UMD spaces is the following.

\begin{Th}\label{UMD}
Let $\B$ be a Banach space. $\B$ is UMD if and only if, for some
(equivalently, for any) $1<p<\infty$ and every $\gamma \in
\mathbb{R}$, the imaginary power $\mathcal{L}^{i\gamma}$ defined on
$L^p(\Omega,\mu)\otimes \B$ can be extended to the Bochner-Lebesgue
space $L^p_\B(\Omega,\mu)$ as a bounded operator from
$L^p_\B(\Omega,\mu)$ into itself.
\end{Th}

The paper is organized as follows. In Section 2 we prove our results
for the Hermite operator. In order to do this
we take advantage from the closed connection existing between the
Ornstein-Uhlenbeck and Hermite settings in the corresponding $L^2$
spaces (see \cite{AT}). The results for Laguerre operators are shown
in Section 3. We exploit some relationships between the Laplace
transform type multipliers in the Laguerre and Hermite contexts.
Estimates shown in Proposition \ref{kernels} are the key of our
technique. There, we compare the kernels $K_\phi ^{L_\alpha}$ and
$K_\phi^H$ in a local region, close to the diagonal, and also find
adequate bounds for both kernels in the global zone, i.e., outside
the local region. This estimations allow us to transfer the results
from the Hermite to the Laguerre cases. This method was first
developed by Betancor et al. in \cite{BFRST}.

Throughout the paper by C and c we represent positive constants that
can change from one line to another. Also, when $\mu=dx$, we simply write $L^p(\Omega)$.

\section{Proof of the results in the Hermite setting}

In this section we show Theorems \ref{Principalvalue} and \ref{UMD} in the Hermite context. The Hermite operator is defined by $\widetilde{H}=-\frac{1}{2}(\frac{d^2}{dx^2}-x^2)$ on $\mathbb{R}$. We have that $\widetilde{H}h_k=\lambda _kh_k$, where, for every $k\in \mathbb{N}$, $\lambda _k=k+1/2$ and $h_k$ denotes the $k$-th Hermite function given by
$$
h_k(x)=(2^k k! \sqrt{\pi })^{-1/2}\mathfrak{H}_k(x)e^{-x^2/2},\quad
x\in \mathbb{R},
$$
$\mathfrak{H}_k$ being the $k$-th Hermite polynomial (\cite{Szeg}). The system $\{h_k\}_{k\in\mathbb{N}}$ is an orthonormal basis in $L^2(\mathbb{R})$. We define the operator $H$ by \eqref{A1.1}. Note that $Hf=\widetilde{H}f$, when $f\in C_c^\infty (\mathbb{R})$, the space of smooth functions with compact support on $\mathbb{R}$.

The semigroup of operators $\{W_t^H\}_{t>0}$ generated by $-H$, defined by \eqref{A2} in $L^2(\mathbb{R})$, can be extended to $L^p(\mathbb{R})$, $1\le p\le\infty$, by (see \cite{ST})
$$
W_t^H(f)(x)=\int_{\mathbb{R}}W_t^H(x,y)f(y)dy,\quad x\in \mathbb{R}\mbox{ and }t>0,
$$
where
$$
W_t^H(x,y)=\frac{1}{\sqrt{\pi }}\left(\frac{e^{-t}}{1-e^{-2t}}\right)^{1/2}e^{-\frac{(x-e^{-t}y)^2+(y-e^{-t}x)^2}{2(1-e^{-2t})}},\quad x,y\in \mathbb{R} \mbox{ and }t>0.
$$

Suppose that $m$ is a function of Laplace transform type given by \eqref{m}. We define the spectral multiplier $T_m^H$ associated with $m$ on $L^2(\mathbb{R})$ by
$$
T_m^H(f)=\sum_{k=0}^\infty m(\lambda _k)c_k^H(f)h_k,
$$
where $\displaystyle c_k^H(f)=\int_\mathbb{R} h_k(y)f(y)dy$, $k\in \mathbb{N}$, and $f\in L^2(\mathbb{R})$.


\subsection{Proof of Theorem \ref{Principalvalue} for the Hermite operator}

We consider the one dimensional Ornstein-Uhlenbeck operator $\widetilde{\mathcal{O}}=-\frac{1}{2}\frac{d^2}{dx^2}+x\frac{d}{dx}$. For every $k\in \mathbb{N}$, $\widetilde{\mathcal{O}}H_k=kH_k$, where $H_k(x)=e^{\frac{x^2}{2}}h_k$, $x\in \mathbb{R}$, is the $k$-th normalized Hermite polynomial (\cite{Szeg}). We define the operator $\mathcal{O}$ by \eqref{A1.1}. The semigroup of operators $\{W_t^\mathcal{O}\}_{t>0}$ generated by $-\mathcal{O}$, defined on $L^2(\mathbb{R},e^{-x^2}dx)$ by \eqref{A2}, is extended to $L^p(\mathbb{R},e^{-x^2}dx)$, $1\le p\le \infty$, by
$$
W_t^\mathcal{O}(f)(x)=\int_\mathbb{R}W_t^\mathcal{O}(x,y)f(y)e^{-y^2}dy,\quad f\in L^p(\mathbb{R}, e^{-x^2}dx)\mbox{ and }t>0,$$
where
$$
W_t^\mathcal{O}(x,y)=\frac{1}{\sqrt{\pi (1-e^{-2t})}}e^{-\frac{(e^{-t}x-y)^2}{1-e^{-2t}}+y^2},\quad x,y\in \mathbb{R}\mbox{ and }t>0.
$$

The spectral multiplier $T_m^\mathcal{O}$ associated to $\mathcal{O}$  and defined by $m$ is
$$
T_m^\mathcal{O}(f)=\sum_{k=0}^\infty m(k)c_k^\mathcal{O}(f)H_k,\quad f\in L^2(\mathbb{R}, e^{-x^2}dx),
$$
where $\displaystyle c_k^\mathcal{O}(f)=\int_\mathbb{R} H_k(y)f(y)e^{-y^2}dy$, $k\in \mathbb{N}$, $f\in L^2(\mathbb{R}, e^{-x^2}dx)$
and we consider $m(0)=0$.

In order to study the operator $T_m^H$ we consider the multiplier $\mathbb{T}_m^\mathcal{O}$ associated with the Ornstein-Uhlenbeck operator defined on $L^2(\mathbb{R}, e^{-x^2}dx)$ by
$$
\mathbb{T}_m^\mathcal{O}(f)=\sum_{k=0}^\infty m\left(k+\frac{1}{2}\right)c_k^\mathcal{O}(f)H_k,\quad f\in L^2(\mathbb{R}, e^{-x^2}dx).
$$

We can write
\begin{align}\label{TTA}
\mathbb{T}_m^\mathcal{O}(f)&=\sum_{k=0}^\infty \left(k+\frac{1}{2}\right)c_k^\mathcal{O}(f)H_k\int_0^\infty e^{-(k+1/2)t}\phi (t)dt\nonumber\\
&=T_M^\mathcal{O}(f)+\frac{1}{2}\sum_{k=0}^\infty c_k^\mathcal{O}(f)H_k\int_0^\infty e^{-(k+1/2)t}\phi (t)dt\nonumber\\
&=T_M^\mathcal{O}(f)+A_\phi (f), \quad f\in L^2(\mathbb{R}, e^{-x^2}dx),
\end{align}
where $M(\lambda)=\lambda \int_0^\infty e^{-\lambda t}e^{-t/2}\phi (t)dt$, $\lambda \in (0,\infty )$, and
$$
A_\phi (f)=\frac{1}{2}\sum_{k=0}^\infty \int_0^\infty e^{-(k+1/2)t}\phi (t)dtc_k^\mathcal{O}(f)H_k, \quad f\in L^2(\mathbb{R}, e^{-x^2}dx).
$$

Since the semigroup $\{W_t^\mathcal{O}\}_{t>0}$ generated by $-\mathcal{O}$ is a symmetric diffusion semigroup, by \cite[Corollary 3, p. 121]{Stei} $T_M^\mathcal{O}$ can be extended to $L^p(\mathbb{R},e^{-x^2}dx)$ as a bounded operator from $L^p(\mathbb{R},e^{-x^2}dx)$ into itself, for every $1<p<\infty$. Moreover, according to \cite[Theorem 3.8]{GCMST}, $T_M^\mathcal{O}$ can be extended to $L^1(\mathbb{R},e^{-x^2}dx)$ as a bounded operator from $L^1(\mathbb{R},e^{-x^2}dx)$ into $L^{1,\infty}(\mathbb{R},e^{-x^2}dx)$.

Let now $f,g\in L^2(\mathbb{R}, e^{-x^2}dx)$. We have that
\begin{align*}
\int_\mathbb{R}A_\phi (f)(x)g(x)e^{-x^2}dx&=\frac{1}{2}\sum_{k=0}^\infty c_k^\mathcal{O}(f)\overline{c_k^\mathcal{O}(g)}\int_0^\infty e^{-(k+1/2)t}\phi (t)dt\\
&=\frac{1}{2}\int_\mathbb{R}\int_0^\infty \phi(t)e^{-t/2}W_t^\mathcal{O}(f)(x)dtg(x)e^{-x^2}dx\\
&=\frac{1}{2}\int_\mathbb{R}g(x) \left(\int_\mathbb{R}\int_0^\infty \phi(t)e^{-t/2}W_t^\mathcal{O}(x,y)dtf(y)e^{-y^2}dy \right)e^{-x^2}dx.
\end{align*}
To justify the above manipulations we observe that
\begin{align*}
\int_0^\infty \int_\mathbb{R}\int_\mathbb{R} & |\phi(t)| e^{-t/2}W_t^\mathcal{O}(x,y)|f(y)|e^{-y^2}|g(x)|e^{-x^2}dxdydt\\
&\leq C \int_0^\infty e^{-t/2} \|W_t^\mathcal{O}(|f|)\|_{L^2(\mathbb{R},e^{-x^2}dx)}
\|g\|_{L^2(\mathbb{R},e^{-x^2}dx)}dt\\
&\leq C \|f\|_{L^2(\mathbb{R}, e^{-x^2}dx)}\|g\|_{L^2(\mathbb{R}, e^{-x^2}dx)}.
\end{align*}

Hence, we get
\begin{equation}\label{8.1}
A_\phi (f)(x)=\int_\mathbb{R}A^\mathcal{O}(x,y)f(y)e^{-y^2}dy,\quad \mbox{ a.e. }x\in \mathbb{R},
\end{equation}
being $A^\mathcal{O}(x,y)=\frac{1}{2}\int_0^\infty \phi(t)W_t^ \mathcal{O}(x,y)e^{-t/2}dt$, $x,y\in \mathbb{R}$.

Note that
$$|A_\phi(f)| \leq C \sup_{t>0} W_t^\mathcal{O}(|f|).$$
Then, by \cite{Mu} and \cite{Sj}, $A_\phi$ is bounded from $L^p(\mathbb{R},e^{-x^2}dx)$ into itself, $1< p<\infty$, and from
$L^1(\mathbb{R},e^{-x^2}dx)$ into $L^{1,\infty}(\mathbb{R},e^{-x^2}dx)$.

Hence, $\mathbb{T }_m^\mathcal{O}$ can be extended to $L^p(\mathbb{R},e^{-x^2}dx)$ as a bounded operator from $L^p(\mathbb{R}, e^{-x^2}dx)$ into itself, for each $1<p<\infty$, and from $L^1(\mathbb{R}, e^{-x^2}dx)$ into $L^{1,\infty }(\mathbb{R}, e^{-x^2}dx)$.

According to \cite[Theorem 4.1]{GCMST} \eqref{TTA} and \eqref{8.1}, there exists $\Lambda\in L^\infty(0,\infty)$ such that
$$
\mathbb{T}_m^\mathcal{O}(f)(x)=\lim_{\varepsilon \rightarrow 0^+}\left(\Lambda (\varepsilon )f(x)+\int_{|x-y|>\varepsilon}(K_\varphi ^\mathcal{O}(x,y)+A^\mathcal{O}(x,y))f(y)e^{-y^2}dy\right),\quad \mbox{ a.e. }x\in \mathbb{R},
$$
and, if the limit $\lim_{t\to 0^+}\phi(t)=\phi(0^+)$ exists, then
$$
\mathbb{T}_m^\mathcal{O}(f)(x)=\phi(0^+)f(x)+\lim_{\varepsilon \rightarrow 0^+}\int_{|x-y|>\varepsilon}(K_\varphi ^\mathcal{O}(x,y)+A^\mathcal{O}(x,y))f(y)e^{-y^2}dy,\quad \mbox{ a.e. }x\in \mathbb{R},
$$
for every $f\in L^p(\mathbb{R},e^{-x^2}dx)$, $1\le p<\infty$, where $\varphi (t)=e^{-t/2}\phi (t)$, $t>0$, and
$$
K_\varphi ^\mathcal{O}(x,y)=\int_0^\infty \varphi (t)\Big(-\frac{\partial }{\partial t}\Big)W_t^\mathcal{O}(x,y)dt,\quad x,y\in \mathbb{R}, \ x\not=y.
$$

We conclude that
\begin{equation}\label{Ornsteinprincipalvalue2}
\mathbb{T}_m^\mathcal{O}(f)(x)=\lim_{\varepsilon \rightarrow 0^+}\left(\Lambda (\varepsilon )f(x)+\int_{|x-y|>\varepsilon}B_\phi ^\mathcal{O}(x,y)f(y)e^{-y^2}dy\right),\quad \mbox{ a.e. }x\in \mathbb{R},
\end{equation}
and, if the limit $\lim_{t\to 0^+}\phi(t)=\phi(0^+)$ exists, then
\begin{equation}\label{Or99}
\mathbb{T}_m^\mathcal{O}(f)(x)=\phi(0^+)f(x)+\lim_{\varepsilon \rightarrow 0^+}\int_{|x-y|>\varepsilon}B_\phi ^\mathcal{O}(x,y)f(y)e^{-y^2}dy,\quad \mbox{ a.e. }x\in \mathbb{R},
\end{equation}
for every $f\in L^p(\mathbb{R},e^{-x^2}dx)$, $1\le p<\infty$, where
$$B_\phi ^\mathcal{O}(x,y)=\int_0^\infty \phi (t)\Big(-\frac{\partial}{\partial t}\Big)(e^{-t/2}W_t^ \mathcal{O}(x,y))dt, \quad x,y\in \mathbb{R}, \ x\not=y.$$

It is clear that, for every $f\in L^2(\mathbb{R})$,
\begin{equation}\label{TmHTmO}
T_m^H(f)(x)=e^{-x^2/2}\mathbb{T}_m^\mathcal{O}(e^{y^2/2}f)(x),\quad \text{a.e. } x\in \mathbb{R}.
\end{equation}

By taking into account that $W_t^H(x,y)=e^{-t/2}e^{-(x^2+y^2)/2}W_t^\mathcal{O}(x,y)$, $x,y\in \mathbb{R}$ and $t>0$, from \eqref{Ornsteinprincipalvalue2}, \eqref{Or99} and \eqref {TmHTmO} we deduce that, for every $f\in L^2(\mathbb{R})$,
\begin{equation}\label{TmHprincipalvalue}
T_m^H(f)(x)=\lim_{\varepsilon \rightarrow 0^+}\left(\Lambda (\varepsilon )f(x)+\int_{|x-y|>\varepsilon}K_\phi ^H(x,y)f(y)dy\right),\quad \mbox{ a.e. }x\in \mathbb{R},
\end{equation}
and, if the limit $\lim_{t\to 0^+}\phi(t)=\phi(0^+)$ exists, then
\begin{equation}\label{Or99H}
T_m^H(f)(x)=\phi(0^+)f(x)+\lim_{\varepsilon \rightarrow 0^+}\int_{|x-y|>\varepsilon}K_\phi^H(x,y)f(y)dy,\quad \mbox{ a.e. }x\in \mathbb{R},
\end{equation}
where
$$K_\phi ^H(x,y)=\int_0^\infty \phi (t)\Big(-\frac{\partial}{\partial t}\Big)W_t^H(x,y)dt, \quad x,y\in \mathbb{R}, \ x\not=y.$$
Hence, if $f\in L^2(\mathbb{R})$,
$$
T_m^H(f)(x)=\int_\mathbb{R}K_\phi ^H(x,y)f(y)dy,\quad \text{a.e. } x\in \mathbb{R}\setminus \mbox{supp}(f).
$$

The kernel function $K_\phi ^H$ is a standard Calder\'on-Zygmund kernel.
\begin{Prop}\label{CZkernel}
Let $\phi \in L^\infty (0,\infty)$. Then, there exists $C>0$ such that
\begin{equation}\label{CZ1}
|K_\phi ^H(x,y)|\leq \frac{C}{|x-y|},\quad x,y\in \mathbb{R}, \ x\not=y,
\end{equation}
and
\begin{equation}\label{CZ2}
\Big|\frac{\partial }{\partial x}K_\phi
^H(x,y)\Big|+\Big|\frac{\partial }{\partial y}K_\phi
^H(x,y)\Big|\leq \frac{C}{|x-y|^2}, \quad x,y\in \mathbb{R}, \
x\not=y.
\end{equation}
\end{Prop}

\begin{proof}
It is not difficult to see (\cite[(2.3)]{BMoR}) that
$$
\left|\frac{\partial }{\partial t}W_t^H(x,y)\right|\leq  Ce^{-\frac{(x-e^{-t}y)^2+(y-e^{-t}x)^2}
{8(1-e^{-2t})}}\frac{e^{-t/2}}{(1-e^{-2t})^{3/2}},\quad x,y\in
\mathbb{R} \mbox{ and }t>0.
$$

Then, by making the change of variables $t=\log(\frac{1+s}{1-s})$ (due to Meda) we get
\begin{align*}
|K_\phi ^H(x,y)|&\leq C\|\phi\|_{L^\infty (0,\infty )}\int_0^1\frac{e^{-\frac{(x-y)^2}{16s}-s\frac{(x+y)^2}{16}}}{s^{3/2}(1-s)^{1/2}}ds\\ &\leq C\|\phi\|_{L^\infty (0,\infty )}\left(\int_0^{1/2}\frac{e^{-\frac{(x-y)^2}{16s}}}{s^{3/2}}ds+\int_{\frac{1}{2}}^1\frac{e^{-\frac{(x-y)^2}{16s}}}{(1-s)^{1/2}}ds\right)\\
&\leq C\|\phi\|_{L^\infty(0,\infty)}\left(\int_0^1\frac{e^{-\frac{(x-y)^2}{16s}}}{s^{3/2}}ds+\frac{1}{|x-y|}\right),\quad x,y\in \mathbb{R}, \ x\not=y.
\end{align*}
In the last inequality we have used that for every $a>0$ and $b\geq 0$, there exists $c>0$ such that $u^be^{-au}\leq c$, $u>0$.
From \cite[Lemma 1.1]{ST} we conclude that
$$
|K_\phi ^H(x,y)|\leq \frac{C}{|x-y|}, \quad x,y\in \mathbb{R}, \ x\not=y.
$$
By proceeding in a similar way \eqref{CZ2} can be shown.
\end{proof}

Calder\'on-Zygmund theory for singular integrals implies that the limits in \eqref{TmHprincipalvalue} and \eqref{Or99H} exist for every $f\in L^p(\mathbb{R})$, $1\le p<\infty$, and $T_m^ H$ can be extended to $L^p(\mathbb{R})$ by \eqref{TmHprincipalvalue} as a bounded operator from $L^p(\mathbb{R})$ into itself, for every $1<p<\infty$, and from $L^1(\mathbb{R})$ into $L^{1,\infty }(\mathbb{R})$.

Thus the proof of Theorem \ref{Principalvalue} for the Hermite operator is finished.

\subsection{Proof of Theorem \ref{UMD} for the Hermite operator}

Let $\B$ be a Banach space and $\phi \in L^\infty(0,\infty)$. By using \eqref{Ornsteinprincipalvalue2}, \eqref{Or99} and \cite[Theorem 4.1]{GCMST}
we can define $T_M^\mathcal{O}$, where $M(\lambda)=\lambda \int_0^\infty e^{-\lambda t}e^{-t/2}\phi(t)dt$, $\lambda>0$, and $\mathbb{T}_m^\mathcal{O}$, in a natural way, on $L^p(\mathbb{R}, e^{-x^2}dx)\otimes \B$, for every $1\leq p<\infty$. Since the operator $A_\phi$ can be extended to $L^p_\B(\mathbb{R},e^{-x^2}dx)$ as a bounded operator from $L^p_\B(\mathbb{R}, e^{-x^2}dx)$ into itself, for every $1< p < \infty$, we can deduce the next property.

\begin{Lem} \label{Lemap}
Let $\B$ be a Banach space, $\phi \in L^\infty(0,\infty)$ and  $1<r<\infty$. The following assertions are equivalent.

(i) $T_M^\mathcal{O}$ can be extended to $L^r_\B(\mathbb{R}, e^{-x^2}dx)$ as a bounded operator from $L^r_\B(\mathbb{R}, e^{-x^2}dx)$ into itself, where $M(\lambda)=\lambda \int_0^\infty e^{-\lambda t}e^{-t/2}\phi(t)dt$, $\lambda>0$.

(ii) $\mathbb{T}_m^\mathcal{O}$ can be extended  to $L^r_\B(\mathbb{R}, e^{-x^2}dx)$ as a bounded operator from $L^r_\B(\mathbb{R}, e^{-x^2}dx)$ into itself, where $m(\lambda)=\lambda \int_0^\infty e^{-\lambda t}\phi(t)dt$, $\lambda>0$.
\end{Lem}

The multiplier $T_m^H$ can be defined on $L^p(\mathbb{R})\otimes \B$, for every $1\leq p<\infty$, in a natural way by using \eqref{TmHprincipalvalue}. Next property follows from  \eqref{TmHTmO}.

\begin{Lem} \label{lema2} Let $\B$ be a Banach space and $m$ a function of Laplace transform type. The following assertions are equivalent.

(i) $T_m^H$ can be extended to $L^2_\B(\mathbb{R})$ as a bounded operator from $L^2_\B(\mathbb{R})$ into itself.

(ii) $\mathbb{T}_m^\mathcal{O}$ can be extended to $L^2_\B(\mathbb{R},e^{-x^2}dx)$ as a bounded operator from $L^2_\B(\mathbb{R},e^{-x^2}dx)$ into itself.
\end{Lem}

Let $\gamma \in \mathbb{R}$. We define $\phi _\gamma (t)=\frac{t^{-i\gamma}}{\Gamma (1-i\gamma )}$, $t\in (0,\infty )$. If $m_\gamma (\lambda)=\lambda \int_0^\infty e^{-\lambda t} \phi _\gamma (t)dt=\lambda ^{i\gamma }$, $\lambda \in (0,\infty )$, we have that
$$
T_{m_\gamma }^H=H^{i\gamma }\quad \mbox{ and }\quad \mathbb{T}_{m_\gamma }^\mathcal{O}=\Big(\mathcal{O}+\frac{1}{2}\Big)^{i\gamma }.
$$

In the following, we characterize UMD spaces by the imaginary powers $H^{i\gamma }$ and $\Big(\mathcal{O}+\frac{1}{2}\Big)^{i\gamma }$.

\begin{Prop}
Let $\B$ a Banach space. The following assertions are equivalent.

(i) $\B$ is UMD.

(ii) For some (equivalently, for every) $1<p<\infty$ and for every $\gamma \in \mathbb{R}$, $\Big(\mathcal{O}+\frac{1}{2}\Big)^{i\gamma }$ can be extended to $L^p_\B(\mathbb{R},e^{-x^2}dx)$ as a bounded operator from $L^p_\B(\mathbb{R},e^{-x^2}dx)$ into itself.

(iii) For some (equivalently, for every) $1<p<\infty$ and for every $\gamma \in \mathbb{R}$, $H^{i\gamma }$ can be extended to $L^p_\B(\mathbb{R})$ as a bounded operator from $L^p_\B(\mathbb{R})$ into itself.
\end{Prop}

\begin{proof}
According to \cite{Guer}, $\B$ is a UMD space if and only if the  imaginary power $L^{i\gamma}$ can be extended to $L^p_\B(\mathbb{R})$ as a bounded operator from $L^p_\B(\mathbb{R})$ into itself for some (equivalently, for any) $1<p<\infty$ and for every $\gamma \in \mathbb{R}$, where $L=-\frac{1}{2}d^2/dx^2$.

\noindent $(i) \iff (ii)$. According to Lemma \ref{Lemap}, in order to prove this it is sufficient to show that if $\gamma \in \mathbb{R}$ and $1<p<\infty$ the following assertions are equivalent:

($i'$) $L^{i\gamma }$ can be extended to $L^p_\B(\mathbb{R})$ as a bounded operator from $L^p_\B(\mathbb{R})$ into itself.

($ii'$) $T_{M_\gamma}^\mathcal{O}$ can be extended to $L^p_\B(\mathbb{R},e^{-x^2}dx)$ as a bounded operator from $L^p_\B(\mathbb{R},e^{-x^2}dx)$ into itself, where $M_\gamma (\lambda )=(\lambda +1/2)^{i\gamma -1}\lambda$, $\lambda \in (0,\infty )$.

Let $1<p<\infty$ and $\gamma \in \mathbb{R}$. Assume that $L^{i\gamma}$ can be extended to $L^p_\B(\mathbb{R})$ as a bounded operator from $L^p_\B(\mathbb{R})$ into itself. We consider a smooth function $\Phi$ on $\mathbb{R}\times \mathbb{R}$ such that $\Phi (x,y)=1$, if $(x,y)\in N_1$; $\Phi (x,y)=0$, $(x,y) \notin N_2$; and
$\Big|\frac{\partial}{\partial x}\Phi(x,y)\Big|+\Big|\frac{\partial}{\partial y}\Phi(x,y)\Big|\leq \frac{C}{|x-y|},$ $x,y\in \mathbb{R}, \ x\not=y$ (see \cite[p. 288]{GCMST}). Here, for every $s>0$, $N_s$ represents the set
$$
N_s=\Big\{(x,y)\in \mathbb{R}^2: |x-y|\leq \frac{s}{1+|x|+|y|}\Big\}.
$$

We  define the operator
$$
L_{\rm loc}^{i\gamma }(g)(x)=L^{i\gamma }(\Phi (x,y)g(y))(x),\quad g\in L^p(\mathbb{R},e^{-x^2}dx)\otimes \B.
$$
The operator $L^{i\gamma }$ satisfies the conditions  \cite[(a), (b) and (c), p. 288]{GCMST}. Then, $L_{\rm loc}^{i\gamma }$ defines a bounded operator from $L^p_\B(\mathbb{R},e^{-x^2}dx)$ into itself (see \cite[Proposition 3.4]{GCMST}).

Also, we consider the operator
$$
T_{M_\gamma ,{\rm loc}}^\mathcal{O}(g)(x)=T_{M_\gamma }^\mathcal{O}(\Phi (x,y)g(y))(x),\quad g\in L^p(\mathbb{R},e^{-x^2}dx)\otimes \B.
$$
We are going to see that the operator $T_{M_\gamma ,{\rm loc}}^\mathcal{O}-L^{i\gamma }_{\rm loc}$ is bounded from $L^p_\B(\mathbb{R},e^{-x^2}dx)$ into itself.

For every $f\in L^p_\B(\mathbb{R},e^{-x^2}dx)$ we know that
\begin{align*}
T_{M_\gamma ,{\rm loc}}^\mathcal{O}(f)(x)-L^{i\gamma }_{\rm loc}(f)(x)&=\lim_{\varepsilon \rightarrow 0^+}\Big(\Lambda (\varepsilon )f(x)\\
&+\int_{|x-y|>\varepsilon }\Phi (x,y)(K_{\varphi_\gamma}
^\mathcal{O}(x,y)e^{-y^2}-K_{\phi_\gamma} (x,y))f(y)dy\Big),\quad
\mbox{a.e. }x\in \mathbb{R},
\end{align*}
for certain $\Lambda \in L^\infty (0,\infty )$, where
$\varphi_\gamma(t)=e^{-t/2}\phi_\gamma(t)$, $t>0$,
$$
K_{\phi_\gamma} (x,y)=\int_0^\infty {\phi_\gamma} (t)\Big(-\frac{\partial}{\partial t}\Big)W_t(x,y)dt,\quad x,y\in \mathbb{R}, \ x\not=y,
$$
and $W_t(x,y)=e^{-\frac{(x-y)^2}{2t}}/\sqrt{2\pi t}$ is the kernel of the heat semigroup associated to $L$.
By making some manipulations we obtain
\begin{align*}
 e^{-y^2} \frac{\partial}{\partial t} W_t^\mathcal{O}(x,y)&=-\frac{1}{\sqrt{\pi }}\frac{e^{-2t}}{(1-e^{-2t})^{3/2}}e^{-\frac{(e^{-t}x-y)^2}{1-e^{-2t}}}\\
&\quad +\frac{2}{\sqrt{\pi }}\frac{e^{-t}(e^{-t}x-y)(x-e^{-t}y)}{(1-e^{-2t})^{5/2}}e^{-\frac{(e^{-t}x-y)^2}{1-e^{-2t}}},\quad x,y\in \mathbb{R}\mbox{ and }t>0,
\end{align*}
and
$$
 \frac{\partial}{\partial t}W_t(x,y)=-\frac{1}{2\sqrt{2\pi }}\frac{e^{-\frac{(x-y)^2}{2t}}}{t^{3/2}}\left(1-\frac{(x-y)^2}{t}\right),\quad x,y\in \mathbb{R}\mbox{ and }t>0.
$$

Note firstly that,
\begin{align*}
\int_{1/|x|^2}^\infty \left|\frac{e^{-t}(e^{-t}x-y)(x-e^{-t}y)}{(1-e^{-2t})^{5/2}}\right|e^{-\frac{(e^{-t}x-y)^2}{1-e^{-2t}}}dt&\\
&\hspace{-6cm}= e^{-\frac{y^2-x^2}{2}}\int_{1/|x|^2}^\infty \left|\frac{e^{-t}(e^{-t}x-y)(x-e^{-t}y)}{(1-e^{-2t})^{5/2}}\right|e^{-\frac{(e^{-t}x-y)^2+(e^{-t}y-x)^2}{2(1-e^{-2t})}}dt\\
&\hspace{-6cm} \leq C e^{-\frac{y^2-x^2}{2}}\int_{1/|x|^2}^\infty \frac{e^{-t}}{(1-e^{-2t})^{3/2}}dt\\
&\hspace{-6cm}\leq Ce^{-\frac{y^2-x^2}{2}}\int_{1/|x|^2}^\infty  \frac{1}{t^{3/2}}dt\leq C|x|e^{-\frac{y^2-x^2}{2}},\quad x,y\in \mathbb{R}.
\end{align*}

Fix $M>0$. If $|x|\geq 1$ and $|x-y|\leq M/|x|$, then $|x|\sim |y|$ and
$$
|y^2-x^2|=|y-x||y+x|\leq \frac{M}{|x|}(|x|+|y|)\leq C.
$$
Also, if $|x|\leq 1$ and $|x-y|\leq M$, then $|y|\leq M+1$. Hence, we deduce that
\begin{align}\label{F1}
\int_{\min\{1,1/|x|^2\}}^\infty \left|\frac{e^{-t}(e^{-t}x-y)(x-e^{-t}y)}{(1-e^{-2t})^{5/2}}\right|e^{-\frac{(e^{-t}x-y)^2}{1-e^{-2t}}}dt
\leq C\max\{1,|x|\},
\end{align}
provided that  $x,y\in \mathbb{R}$ and $|x-y|\leq M\min\{1,1/|x|\}$.

We can write, for every $x,y\in \mathbb{R}$,
\begin{equation}\label{F2}
\int_{\min\{1,1/|x|^2\}}^\infty \frac{e^{-2t}}{(1-e^{-2t})^{3/2}}e^{-\frac{(e^{-t}x-y)^2}{1-e^{-2t}}}dt\leq C\int_{\min\{1,1/|x|^2\}}^\infty \frac{1}{t^{3/2}}dt\leq C\max\{1,|x|\},
\end{equation}
and
\begin{equation}\label{F3}
\int_{\min\{1,1/|x|^2\}}^\infty \left|\frac{\partial}{\partial t}W_t(x,y)\right|dt\leq C\int_{\min\{1,1/|x|^2\}}^\infty \frac{1}{t^{3/2}}dt\leq
C\max\{1,|x|\},\quad x,y\in \mathbb{R}.
\end{equation}

On the other hand, we have that
\begin{align}\label{F4}
\left|e^{-\frac{(e^{-t}x-y)^2}{1-e^{-2t}}}-e^{-\frac{(x-y)^2}{2t}}\right|&=e^{-\frac{(x-y)^2}{2t}}\left|e^{\frac{(x-y)^2}{2t}-\frac{(e^{-t}x-y)^2}{1-e^{-2t}}}-1\right|\nonumber\\
&\leq Ce^{-\frac{(x-y)^2}{2t}}\left|\frac{(x-y)^2}{2t}-\frac{(e^{-t}x-y)^2}{1-e^{-2t}}\right|\nonumber\\
&\leq Ce^{-\frac{(x-y)^2}{2t}}\left((x-y)^2\left|\frac{1}{2t}-\frac{1}{1-e^{-2t}}\right|+\frac{|(x-y)^2-(e^{-t}x-y)^2|}{1-e^{-2t}}\right)\nonumber\\
&\leq Ce^{-\frac{(x-y)^2}{2t}}\left((x-y)^2+\frac{(1-e^{-t})|x||x-y|+(1-e^{-t})^2x^2}{1-e^{-2t}}\right)\nonumber\\
&\leq Ce^{-\frac{(x-y)^2}{2t}}((x-y)^2+|x||x-y|+t|x|^2),\quad x,y\in \mathbb{R}, \ t>0.
\end{align}

By using \eqref{F4} and \cite[(B), p. 15]{HTV} we get
\begin{align*}
\left|\frac{e^{-t/2}e^{-2t}}{(1-e^{-2t})^{3/2}}e^{-\frac{(e^{-t}x-y)^2}{1-e^{-2t}}}-\frac{1}{(2t)^{3/2}}e^{-\frac{(x-y)^2}{2t}}\right|&\\
&\hspace{-5cm}\leq
\left|\frac{e^{-t/2}e^{-2t}}{(1-e^{-2t})^{3/2}}-\frac{1}{(2t)^{3/2}}\right|e^{-\frac{(e^{-t}x-y)^2}{1-e^{-2t}}}
+\frac{1}{(2t)^{3/2}}\left|e^{-\frac{(e^{-t}x-y)^2}{1-e^{-2t}}}-e^{-\frac{(x-y)^2}{2t}}\right|\\
&\hspace{-5cm}\leq Ce^{-c\frac{(x-y)^2}{t}}
\left(\frac{1}{t^{1/2}}+\frac{(x-y)^2+|x||x-y|+t|x|^2}{t^{3/2}}\right),\quad
(x,y)\in N_2, \ t>0.
\end{align*}
Hence, by using \cite[Lemma 1.1]{ST}, it follows that
\begin{align}\label{F5}
\int_0^{\min \{1,1/|x|^2\}}\left|\frac{e^{-t/2}e^{-2t}}{(1-e^{-2t})^{3/2}}e^{-\frac{(e^{-t}x-y)^2}{1-e^{-2t}}}-\frac{1}{(2t)^{3/2}}e^{-\frac{(x-y)^2}{2t}}\right|dt&\nonumber\\
&\hspace{-7cm}\leq C\int_0^{\min \{1,1/|x|^2\}}e^{-c\frac{(x-y)^2}{t}}\left(\frac{1}{t^{1/2}}+\frac{|x||x-y|}{t^{3/2}}+\frac{|x|^2}{t^{1/2}}\right)\leq C\max\{1,|x|\},\quad x,y\in \mathbb{R}.
\end{align}

Also we have that
\begin{align}\label{F6}
&\left|\frac{(x-y)^2}{(2t)^{5/2}}-\frac{e^{-t/2}e^{-t}(e^{-t}x-y)(x-e^{-t}y)}{(1-e^{-2t})^{5/2}}\right|\nonumber\\
& \quad \leq C\left(
\frac{|e^{-t}x-y||x-e^{-t}y|}{t^{3/2}}+\frac{|x||x-e^{-t}y|}{t^{3/2}}+\frac{|y||e^{-t}x-y|}{t^{3/2}}
+\left|\frac{1}{(2t)^{5/2}}-\frac{1}{(1-e^{-2t})^{5/2}}\right|(x-y)^2\right)\nonumber\\
&\quad \leq
\frac{C}{t^{3/2}}\left(|e^{-t}x-y||x-e^{-t}y|+|x||x-e^{-t}y|+|y||e^{-t}x-y|+(x-y)^2\right),\quad
x,y\in \mathbb{R}\mbox{ and } 0<t<1.
\end{align}

The estimations \eqref{F4}, \eqref{F6} and \cite[(B), p. 15]{HTV} lead to
\begin{align*}
&\left|\frac{e^{-t/2}e^{-t}(e^{-t}x-y)(x-e^{-t}y)}{(1-e^{-2t})^{5/2}}e^{-\frac{(e^{-t}x-y)^2}{1-e^{-2t}}}-\frac{(x-y)^2}{(2t)^{5/2}}e^{-\frac{(x-y)^2}{2t}}\right|\\
& \quad \leq C\left|\frac{e^{-t/2}e^{-t}(e^{-t}x-y)(x-e^{-t}y)}{(1-e^{-2t})^{5/2}}-\frac{(x-y)^2}{(2t)^{5/2}}\right|e^{-\frac{(e^{-t}x-y)^2}{1-e^{-2t}}}\\
&\qquad +\frac{(x-y)^2}{(2t)^{5/2}}\left|e^{-\frac{(e^{-t}x-y)^2}{1-e^{-2t}}}-e^{-\frac{(x-y)^2}{2t}}\right|\\
&\quad\leq\frac{C}{t^{3/2}}\left[(|e^{-t}x-y||x-e^{-t}y|+|x||x-e^{-t}y|+|y||e^{-t}x-y|)e^{-\frac{(x-e^{-t}y)^2+(e^{-t}x-y)^2}{2(1-e^{-2t})}-\frac{y^2-x^2}{2}}\right.\\
& \qquad\left.+(x-y)^2e^{-c\frac{(x-y)^2}{t}}+\frac{(x-y)^2}{t}e^{-\frac{(x-y)^2}{2t}}((x-y)^2+|x||x-y|+t|x|^2)\right]\\
& \quad \leq
C\left(\frac{1+|x|^2+|xy|}{t^{1/2}}+\frac{(|x|+|y|)|x-y|}{t^{3/2}}e^{-c\frac{(x-y)^2}{t}}\right),\quad
(x,y)\in N_2 \mbox{ and }0<t<1.
\end{align*}

Then, by using again \cite[Lemma 1.1]{ST} we obtain
\begin{align}\label{F7}
\int_0^{\min\{1,1/|x|^2\}} \left|\frac{e^{-t/2}e^{-t}(e^{-t}x-y)(x-e^{-t}y)}{(1-e^{-2t})^{5/2}}e^{-\frac{(e^{-t}x-y)^2}{1-e^{-2t}}}-\frac{(x-y)^2}{(2t)^{5/2}}e^{-\frac{(x-y)^2}{2t}}\right|\nonumber\\
&\hspace{-8cm} \leq C\int_0^{\min\{1,1/|x|^2\}} \left(\frac{1+x^2+|xy|}{t^{1/2}}+\frac{(|x|+|y|)|x-y|}{t^{3/2}}e^{-c\frac{(x-y)^2}{t}}\right)dt\nonumber\\
&\hspace{-8cm}\leq C\max\{1,|x|\}, \quad (x,y) \in N_2 \mbox{ and
}t>0.
\end{align}

By combining \eqref{F1}, \eqref{F2}, \eqref{F3}, \eqref{F5} and \eqref{F7}, since there exists $M>0$ such that $N_2\subset\{(x,y)\in \mathbb{R}^2:|x-y|\leq M\min\{1,\frac{1}{|x|}\}\}$,
$$
|\Phi(x,y)| \ |K_{\varphi_\gamma}^\mathcal{O}(x,y)e^{-y^2}-K_{\phi_\gamma} (x,y)|\leq
C\max\{1,|x|\}\chi _{N_2}(x,y),\quad x,y\in \mathbb{R}.
$$

It is not hard to see that
$$
\sup_{x\in \mathbb{R}}\max\{1,|x|\} \int_\mathbb{R}\chi _{N_2}(x,y)dy<\infty \quad \mbox{ and }\quad
\sup_{y\in \mathbb{R}}\int _\mathbb{R}\max\{1,|x|\}\chi _{N_2}(x,y)dx<\infty .
$$

According to \cite[Lemma 3.6]{GCMST}, $T_{M_\gamma ,{\rm loc}}^\mathcal{O}-L_{\rm loc}^{i\gamma }$ is a bounded operator from $L^p_\B(\mathbb{R},e^{-x^2}dx)$ into itself. Hence, $T_{M_\gamma ,{\rm loc}}^\mathcal{O}$ is bounded from $L^p_\B(\mathbb{R},e^{-x^2}dx)$ into itself.

On the other hand, by proceeding as in the proof of \cite[Theorem 3.8]{GCMST} we have
$$
\|T_{M_\gamma,{\rm glob}}^\mathcal{O}(f)(x)\|_\B\leq C\int_\mathbb{R} \sup_{t>0}|W_t^\mathcal{O}(x,y)|(1-\chi _{N_1}(x,y))\|f(y)\|_\B dy,
$$
and from \cite{Sj} (see also \cite{MPS}) we deduce that
$T_{M_\gamma ,{\rm glob}}^\mathcal{O}$ is bounded from
$L^p_\B(\mathbb{R},e^{-x^2}dx)$ into itself.

We conclude that $T_{M_\gamma }^\mathcal{O}$ can be extended to $L^p_\B(\mathbb{R},e^{-x^2}dx)$ as a bounded operator from $L^p_\B(\mathbb{R},$ $e^{-x^2}dx)$ into itself,  and $(i')\Longrightarrow (ii')$ is proved.

We are going to show that $(ii')\Longrightarrow (i')$. Let $\gamma \in \mathbb{R}$ and $1<p<\infty $. Assume that $T_{M_\gamma }^\mathcal{O}$ can be extended to $L^p_\B(\mathbb{R},e^{-x^2}dx)$ as a bounded operator from $L^p_\B(\mathbb{R},e^{-x^2}dx)$ into itself. The arguments developed above imply that the operator $L_{\rm loc} ^{i\gamma}$ is bounded from $L^p_\B(\mathbb{R},e^{-x^2}dx)$ into itself. According to \cite[Lemma 3.6]{GCMST} and \cite[Propositions 2.3 and 2.4]{HTV}, that also holds from Banach valued operators, $L_{\rm loc} ^{i\gamma}$ is bounded from $L^p_\B(\mathbb{R})$ into itself.

In order to see that $L^{i\gamma }$ is bounded from  $L^p_\B(\mathbb{R})$  into itself we use the ideas presented in \cite[p. 21]{HTV}. Let $f\in C_c^\infty  (\mathbb{R})\otimes \B$. For every $R>0$ we define $f_R(x)=f(Rx)$, $x\in\mathbb{R}$. Mote that, for every $R,\varepsilon >0$,
$$
\int_{|\frac{x}{R}-y|>\varepsilon }K_{\phi _\gamma }\Big(\frac{x}{R},y\Big)f_R(y)dy=R\int_{|\frac{x-z}{R}|>\varepsilon }K_{\phi _\gamma }\Big(\frac{x}{R},\frac{z}{R}\Big)f(z)dz,\quad x\in \mathbb{R},
$$
and, if $a>0$ there exists $R_a>0$ such that
$$
\left|\frac{x}{R}-\frac{z}{R}\right|\leq \frac{1}{1+|x|+|z|},\quad |x|\leq a, |z|\leq a\mbox{ and }R\geq R_a.
$$

Let $a>0$. There exits $R_a\in \mathbb{N}$ such that, for every $R\geq R_a$, $R\in \mathbb{N}$,
\begin{align*}
L^{i\gamma }(f_R)\Big(\frac{x}{R}\Big)&=\lim_{\varepsilon \rightarrow 0^+}\left(\Lambda (\varepsilon )f_R\Big(\frac{x}{R}\Big)+\int_{|\frac{x}{R}-y|>\varepsilon }
K_{\phi _\gamma }\Big(\frac{x}{R},y\Big)\Phi \Big(\frac{x}{R},y\Big)f_R(y)dy\right)\\
&=L^{i\gamma }_{\rm loc}(f_R)\Big(\frac{x}{R}\Big),\quad \mbox{ a.e. }x\in (0,a),
\end{align*}
for a certain $\Lambda\in L^\infty(0,\infty)$.

Moreover, by extending the Fourier transformation to $C_c^\infty (\mathbb{R})\otimes \B$ in a natural way, we can write, for every $R>0$,
\begin{align*}
L^{i\gamma }(f_R)\Big(\frac{x}{R}\Big)=(|y|^{2i\gamma} \hat{f_R}(y))^\vee\Big(\frac{x}{R}\Big)=R(|Ry|^{2i\gamma }\hat{f_R}(Ry))^\vee(x)
=R^{2i\gamma}L^{i\gamma }(f)(x),\quad \mbox{ a.e. }x\in (0,a).
\end{align*}
Hence, we have that
\begin{align*}
\int_{-a}^a\|L^{i\gamma }(f)(x)\|_\B^pdx
    &=\int_{-a}^a\|L^{i\gamma }(f_R)\Big(\frac{x}{R}\Big)\|_\B^pdx
     =\int_{-a}^a\|L_{\rm loc}^{i\gamma }(f_R)\Big(\frac{x}{R}\Big)\|_\B^pdx\\
    &\leq R\int_{-\infty}^{+\infty}\|L_{\rm loc}^{i\gamma }(f_R)(x)\|_\B^pdx\leq CR\int_{-\infty }^{+\infty }\|f(Ry)\|_\B^pdy
    \leq C\int_{-\infty }^{+\infty}\|f(y)\|_\B^pdy.
\end{align*}
By taking $a\rightarrow\infty $ we conclude that $L^{i\gamma }$ can be extended to $L_\B^p(\mathbb{R})$ as a bounded operator from $L_\B^p(\mathbb{R})$ into itself, and $(i')$ is proved.

$(ii)\iff (iii)$. Assume that the operator $H^{i\gamma}$ can be extended to $L_\B^p(\mathbb{R})$ as a bounded operator from $L_\B^p(\mathbb{R})$ into itself, for some $1<p<\infty$. Then, since $K_\phi ^H$ is  a Calder\'on-Zygmund kernel (Proposition~\ref{CZkernel}), it follows that $H^{i\gamma}$ can be extended to $L_\B^q(\mathbb{R})$ as a bounded operator from $L_\B^q(\mathbb{R})$ into itself, for every $1<q<\infty$. Hence, the equivalence $(ii)\iff (iii)$ follows from Lemma \ref{lema2}.

\end{proof}

\section{Proof of the results in the Laguerre settings}

Let $\alpha >-1/2$. We consider the Laguerre differential operator $\widetilde{L}_\alpha$ given by
$$
\widetilde{L}_\alpha =-\frac{1}{2}\left(\frac{d^2}{dx^2}-x^2-\frac{\alpha ^2-1/4}{x^2}\right),\quad x\in (0,\infty ).
$$
Fore every $k\in \mathbb{N}$, we have that $\widetilde{L}_\alpha \varphi_k^\alpha=\lambda_k^\alpha\varphi_k^\alpha$, where $\lambda_k^\alpha=2k+\alpha+1$ and $\varphi_k^\alpha$ denotes the $k$-th Laguerre function defined by
$$
\varphi _k^\alpha (x)=\left(\frac{2\Gamma (k+1)}{\Gamma (k+\alpha +1)}\right)^{1/2}e^{-x^2/2}x^{\alpha +1/2}L_k^\alpha (x^2),\quad x\in (0,\infty ),
$$
$L_k^\alpha$ being the $k$-th Laguerre polynomial of type $\alpha$ (see \cite[p. 100]{Szeg} and \cite[p. 7]{Than}). The system $\{\varphi _k^\alpha \}_{k\in \mathbb{N}}$ is an orthonormal basis for $L^2(0,\infty )$. The operator $L_\alpha$ is defined by \eqref{A1.1}. Note that if $f\in C_c^\infty(0,\infty)$, the space of the smooth functions with compact support in $(0,\infty)$, then $L_\alpha f=\widetilde{L}_\alpha f$.

The heat semigroup $\{W_t^{L_\alpha}\}_{t>0}$ generated by $-L_\alpha$ is defined in $L^2(0,\infty)$ by \eqref{A2}. By using the Mehler formula (\cite[p. 8]{Than}), for every $t>0$ and $1\le p\le\infty$, the operator $W_t^{L_\alpha}$ can be extended to $L^p(0,\infty)$ as a bounded operator from $L^p(0,\infty)$ into itself defining, for each $f\in L^p(0,\infty)$,
$$
W_t^{L_\alpha} (f)(x)=\int_0^\infty W_t^{L_\alpha} (x,y)f(y)dy,\quad x\in (0,\infty ),
$$
where
$$
W^{L_\alpha}_{t}(x,y)=\Biggl(\frac{2e^{-t}}{1-e^{-2t}}\Biggr)^{1/2}\Biggl(\frac{2xye^{-t}}{1-e^{-2t}}\Biggr)^{1/2}
I_{\alpha}\Biggl(\frac{2xye^{-t}}{1-e^{-2t}}\Biggr)e^{-\frac{1}{2}(x^{2}+y^{2})\frac{1+e^{-2t}}{1-e^{-2t}}}.
$$
Here $I_{\alpha}$ denotes the modified Bessel function of the
first kind and order $\alpha$.

As it was mentioned in the introduction, in order to prove Theorems \ref{Principalvalue} and \ref{UMD} for the Laguerre operators we exploit some connections between Hermite and Laguerre settings. This link is shown in the property $(d)$ of the following proposition. Next, we establish estimates for the kernels $K_\phi^H$ and $K_\phi^{L_\alpha}$ that are crucial in the proof of the results.

\begin{Prop}\label{kernels}
Let $\phi \in L^\infty (0,\infty )$. Then, there exists $C>0$ such that

(a) $|K_\phi ^H(x,y)|\leq C\frac{\|\phi\|_{L^\infty (0,\infty )}}{\max \{x,y\}}$, $0<y<\frac{x}{2}$ or $0<2x<y$.

(b) $|K_\phi ^H(x,y)|\leq C\frac{\|\phi\|_{L^\infty (0,\infty )}}{|x|+|y|}$, $xy<0$.

(c) $|K_\phi ^{L_\alpha} (x,y)|\leq C\|\phi\|_{L^\infty (0,\infty )}\frac{(\min\{x,y\})^{\alpha +1/2}}{(\max \{x,y\})^{\alpha+3/2}}$, $0<y<\frac{x}{2}$ or $0<2x<y$.

(d)  $|K_\phi ^{L_\alpha} (x,y)-K_\phi ^H(x,y)|\leq C\frac{\|\phi\|_{L^\infty (0,\infty )}}{x}\left(1+\sqrt{\frac{x}{|y-x|}}\right)$, $0<\frac{x}{2}<y<2x$.
\end{Prop}
\begin{proof}
(a) We  have only to take into account \eqref{CZ1} and observe that $|x-y|\sim y$, when $y>2x>0$, and that $|x-y|\sim x$, provided that $0<y<\frac{x}{2}$.

(b) It is a direct consequence of  \eqref{CZ1}.

(c) Since $K_\phi ^{L_\alpha} (x,y)=K_\phi ^{L_\alpha} (y,x)$, $x,y\in (0,\infty )$, it is sufficient to establish (c) when $0<y<\frac{x}{2}$.

Assume $0<y<\frac{x}{2}$, and denote by $L(x,y)$ and $R(x,y)$ the following sets:
$$
L(x,y)=\left\{t>0: \frac{e^{-t}xy}{1-e^{-2t}}\leq 1\right\}, \quad
R(x,y)=\left\{t>0: \frac{e^{-t}xy}{1-e^{-2t}}\geq 1\right\}.
$$
These regions play an important role when estimating the modified Bessel function $I_\alpha$.

By using \cite[(3.4) and (3.5)]{BMoR} the asymptotics for $I_\alpha$ allow us to write
\begin{align}\label{K}
|K_\phi ^{L_\alpha} (x,y)|&\leq \|\phi\|_{L^\infty (0,\infty )}\left(\int_{L(x,y)}+\int_{R(x,y)}\right)\left|\frac{\partial }{\partial t}W_t^{L_\alpha} (x,y)\right|dt\nonumber\\
&\leq C\left((xy)^{\alpha +1/2}\int_0^\infty \frac{e^{-\frac{x^2}{8t}}e^{-(\alpha +1)t}}{(1-e^{-2t})^{\alpha +2}}dt+x^2\int_{R(x,y)}\frac{e^{-\frac{3t}{2}}e^{-\frac{(x-e^{-t}y)^2+(y-e^{-t}x)^2}{2(1-e^{-2t})}}}{(1-e^{-2t})^{5/2}}
dt\right).
\end{align}

To estimate the first integral we observe that $t\sim 1-e^{-2t}$, as $t\rightarrow 0$; that, for certain $c>0$, $1-e^{-2t}\geq c$, when $t>1$; and that, as it was mentioned before, for every $a>0$ and $b\geq 0$, there exists $c>0$ for which $u^be^{-au}\leq c$, $u>0$. Thus, by using again \cite[Lemma 1.1]{ST},
\begin{equation}\label{2.1.a}
\int_0^\infty \frac{e^{-\frac{x^2}{8t}}e^{-(\alpha +1)t}}{(1-e^{-2t})^{\alpha +2}}dt\leq C\left(\int_0^1\frac{e^{-\frac{x^2}{8t}}}{t^{\alpha +2}}dt+\frac{1}{x^{2\alpha +2}}\int_1^\infty e^{-(\alpha +1)t}t^{\alpha +1}dt\right)\leq \frac{C}{x^{2\alpha +2}}.
\end{equation}

On the other hand, by making the change of variables $t=\log(\frac{1+s}{1-s})$ in the second integral of \eqref{K} and taking into account that $\alpha +\frac{1}{2}>0$ we get
\begin{eqnarray*}
\int_{R(x,y)}\frac{e^{-\frac{3t}{2}}e^{-\frac{(x-e^{-t}y)^2+(y-e^{-t}x)^2}{2(1-e^{-2t})}}}{(1-e^{-2t})^{5/2}}dt&\leq& C\int_{0,\frac{(1-s^2)xy}{4s}\geq 1}^1\frac{e^{-\frac{(x-y)^2}{4s}}(1-s)^{1/2}}{s^{5/2}}ds\\
&\hspace{-4cm}\leq&\hspace{-2cm}C(xy)^{\alpha +1/2}\int_0^1\frac{e^{-\frac{(x-y)^2}{4s}}}{s^{\alpha +3}}ds\leq C\frac{(xy)^{\alpha +1/2}}{|x-y|^{2\alpha +4}}\leq C\frac{y^{\alpha+1/2}}{x^{\alpha +7/2}}.
\end{eqnarray*}

Here we have again made use of \cite[Lemma 1.1]{ST} and that $|x-y|\sim x$, when $0<y<\frac{x}{2}$.

(d) Let $0<\frac{x}{2}<y<2x$, and $L(x,y)$ and $R(x,y)$ as before. By \cite[(3.4),(3.6) and (3.13)]{BMoR} we can write
\begin{align*}
|K_\phi ^{L_\alpha} (x,y)-K_\phi ^H(x,y)|\leq &\|\phi \|_{L^\infty (0,\infty )}\left(\int_{L(x,y)}\left(\left|\frac{\partial }{\partial t}W_t^{L_\alpha} (x,y)\right|+\left|\frac{\partial }{\partial t}W_t^H(x,y)\right|\right)dt  \right.\\
& +\left.\int_{R(x,y)}\left|\frac{\partial }{\partial t}W_t^{L_\alpha} (x,y)-\frac{\partial }{\partial t}W_t ^H(x,y)\right|dt\right)\\
\leq& C\|\phi \|_{L^\infty (0,\infty )}\left((xy)^{\alpha +1/2}\int_0^\infty \frac{e^{-\frac{x^2}{8t}}e^{-(\alpha +1)t}}{(1-e^{-2t})^{\alpha +2}}dt
+\int_0^\infty \frac{e^{-\frac{x^2}{8t}}e^{-t/2}}{(1-e^{-2t})^{3/2}}dt\right.\\
& +\left.\int_{R(x,y)}\frac{e^{-\frac{(x-y)^2}{2t}}e^{t/2}}{xy\sqrt{1-e^{-t}}}dt\right).
\end{align*}

The two first integrals can be estimated by proceeding as in (c) (see \eqref{2.1.a}). Observe that the second one is exactly the first one for $\alpha =-1/2$. On the other  hand,
\begin{align*}
\int_{R(x,y)}\frac{e^{-\frac{(x-y)^2}{2t}}e^{t/2}}{xy\sqrt{1-e^{-2t}}}dt\leq &\frac{C}{(xy)^{1/4}}\int_{R(x,y)}\frac{e^{-\frac{(x-y)^2}{2t}}e^{-t/4}}{(1-e^{-2t})^{5/4}}dt\\
\leq &\frac{C}{\sqrt{x}}\left(\int_0^1\frac{e^{-\frac{(x-y)^2}{2t}}}{t^{5/4}}dt+\int_1^\infty \frac{e^{-t/4}t^{1/4}}{\sqrt{|x-y|}}dt\right)\\
\leq &\frac{C}{\sqrt{x|x-y|}}=\frac{C}{x}\sqrt{\frac{x}{|x-y|}},\quad x\not =y.
\end{align*}
Thus (d) is established.
\end{proof}

\subsection{Proof of Theorem \ref{Principalvalue} for Laguerre operators}

Let $f,g\in C_c^\infty (0,\infty )$. The spectral multiplier $T_m^{L_\alpha}$ is a bounded operator on $L^2(0,\infty)$. Then, we can write
$$
\langle T_m^{L_\alpha} (f),g\rangle _{L^2(0,\infty)}=\sum_{k=0}^\infty m(\lambda _k^\alpha )c_k^\alpha (f)\langle \varphi _k^\alpha ,g\rangle _{L^2(0,\infty )}=\sum_{k=0}^\infty \lambda _k^\alpha c_k^\alpha (f)\overline{c_k^\alpha (g)}\int_0^\infty e^{-\lambda _k^\alpha t}\phi (t)dt,
$$
where $\displaystyle c_k^\alpha(h)=\int_0^\infty \varphi_k^\alpha(x)h(x)dx$, $h\in L^2(0,\infty)$ and $k\in \mathbb{N}$. Since
\begin{align*}
\int_0^\infty \Big|\phi (t)\sum_{k=0}^\infty \lambda _k^\alpha e^{-\lambda _k^\alpha t}c_k^\alpha (f)\overline{c_k^\alpha (g)}\Big|dt \leq&
\|\phi \|_{L^\infty(0,\infty )}\sum_{k=0}^\infty |c_k^\alpha (f)||c_k^\alpha (g)|\int_0^\infty \lambda _k^\alpha e^{-\lambda _k^\alpha t}dt\\
\leq&\|\phi \|_{L^\infty(0,\infty )}\|f\|_{L^2(0,\infty )}\|g\|_{L^2(0,\infty )},
\end{align*}
we obtain
\begin{align*}
\langle T_m^{L_\alpha} (f),g\rangle _{L^2(0,\infty )}=&\int_0^\infty \phi (t)\sum_{k=0}^\infty \lambda_k^\alpha e^{-\lambda _k^\alpha t}c_k^\alpha (f)\overline{c_k^\alpha (g)}dt\\
=&\int_0^\infty \phi (t)\Big(-\frac{d}{dt}\Big)\sum_{k=0}^\infty e^{-\lambda _k^\alpha t}c_k^\alpha (f)\overline{c_k^\alpha (g)}dt.
\end{align*}
To justify the last equality let us denote by $\Phi (t)=\sum_{k=0}^\infty e^{-\lambda _k^\alpha t}c_k^\alpha (f)\overline{c_k^\alpha (g)}$, $t>0$. Observe that, for every $t>0$,
$$
|\Phi (t)|\leq \sum_{k=0}^\infty |c_k^\alpha (f)||c_k^\alpha (g)|\leq \|f\|_{L^2(0,\infty )}\|g\|_{L^2(0,\infty )}<\infty .
$$
Let $t>0$. For every $|h|<\frac{t}{2}$ we get that
\begin{eqnarray*}
\left|\frac{\Phi (t+h)-\Phi (t)}{h}+\sum_{k=0}^\infty \lambda _k^\alpha e^{-\lambda _k^\alpha t}c_k^\alpha (f)\overline{c_k^\alpha (g)}\right|&&\\
&\hspace{-8cm}\leq&\hspace{-4cm}\sum_{k=0}^\infty \left(\frac{e^{-\lambda _k^\alpha h}-1}{h}+\lambda _k^\alpha \right)e^{-\lambda _k^\alpha t}|c_k^\alpha (f)||c_k^\alpha (g)|\\
&\hspace{-8cm}\leq&\hspace{-4cm}|h|\sum_{k=0}^\infty(\lambda_k^\alpha )^2e^{-\lambda _k^\alpha (t-|h|)}|c_k^\alpha (f)||c_k^\alpha (g)|\\
&\hspace{-8cm}\leq&\hspace{-4cm}C\frac{|h|}{(t-|h|)^2}\sum_{k=0}^\infty|c_k^\alpha (f)||c_k^\alpha (g)|\leq C\frac{|h|}{t^2} \|f\|_{L^2(0,\infty )}\|g\|_{L^2(0,\infty )},
\end{eqnarray*}
which leads to
$$
\sum_{k=0}^\infty \lambda _k^\alpha e^{-\lambda _k^\alpha t}c_k^\alpha (f)\overline{c_k^\alpha (g)}=\Big(-\frac{d}{dt}\Big)\sum_{k=0}^\infty e^{-\lambda _k^\alpha t}c_k^\alpha (f)\overline{c_k^\alpha (g)}.
$$
Thus we obtain that
$$
\langle T_m^{L_\alpha} (f),g\rangle _{L^2(0,\infty)}=\int_0^\infty \phi (t)\Big(-\frac{d}{dt}\Big)\langle W_t^{L_\alpha} (f),g\rangle _{L^2(0,\infty )}dt.
$$

From now on, if $F$ is a function defined on $(0,\infty )$, we denote by $\widetilde{F}$ the function on $\mathbb{R}$ given by $\widetilde{F}(x)=F(x)$, $x>0$, and $\widetilde{F}(x)=0$, $x\leq 0$. Then, we write
\begin{align*}
\langle T_m^{L_\alpha} (f),g\rangle _{L^2(0,\infty)}=&\int_0^\infty \phi (t)\Big(-\frac{d}{dt}\Big)\langle W_t^{L_\alpha} (f)-W_{t,+}^H(f),g\rangle _{L^2(0,\infty)}dt\\
+&\int_0^\infty \phi (t)\Big(-\frac{d}{dt}\Big)\langle W_{t,+}^H(f),g\rangle _{L^2(0,\infty)}dt,
\end{align*}
where $W_{t,+}^H(f)=W_t^H(\tilde{f})$. Note that $\langle W_{t,+}^H(f),g\rangle _{L^2(0,\infty)}=\langle W_t^H(\tilde{f}),\tilde{g}\rangle _{L^2(\mathbb{R})}$. Then
\begin{align*}
\langle T_m^{L_\alpha} (f),g\rangle _{L^2(0,\infty)}&=\int_0^\infty \phi (t)\Big(-\frac{d}{dt}\Big)\langle W_t^{L_\alpha} (f)-W_{t,+}^H(f),g\rangle _{L^2(0,\infty )}dt\\
&+\int_0^\infty \phi (t)\Big (-\frac{d}{dt}\Big)\langle W_t^H(\tilde{f}),\tilde{g}\rangle _{L^2(\mathbb{R})}dt.
\end{align*}
By \eqref{TmHprincipalvalue} we can find $\Lambda \in L^\infty (0,\infty)$ such that
\begin{align}\label{A1}
\langle T_m^{L_\alpha} (f),g\rangle _{L^2(0,\infty)}=&\int_0^\infty \phi (t)\Big(-\frac{d}{dt}\Big)\langle W_t^{L_\alpha} (f)-W_{t,+}^H(f),g\rangle _{L^2(0,\infty)}dt\nonumber\\
+&\Big\langle \lim_{\varepsilon \rightarrow 0^+}\Big(\Lambda (\varepsilon )\tilde{f}(x)+\int_{|x-y|>\varepsilon }\tilde{f}(y)K_\phi ^H(x,y)dy\Big),\tilde{g}\Big\rangle _{L^2(\mathbb{R})}.
\end{align}
Moreover, is there exists the limit $\lim_{t\to 0^+}\phi(t)=\phi(0^+)$, then $\lim_{\varepsilon\to 0^+}\Lambda(\varepsilon)=\phi(0^+)$.

Since $f\in C_c^\infty (0,\infty )$, by using partial integration and by taking into account well known properties of Hermite and Laguerre functions (\cite{Than}), we can show that, for every $k\in \mathbb{N}$ there exists $C_k>0$ such that
$$
|c_m^\alpha (f)|\leq \frac{C_k}{(m+1)^k}\quad \mbox{ and }\quad |c_m(\tilde{f})|\leq \frac{C_k}{(m+1)^k},\quad m\in \mathbb{N},
$$
where $\displaystyle c_m(f)=\int_\mathbb{R} h_m(x) f(x) dx$, $m\in
\mathbb{N}$. We deduce that
\begin{align*}
\int_0^\infty & |\phi (t)|\int_0^\infty \left|\frac{\partial}{\partial t}(W_t^{L_\alpha} (f)(x)-W_t^H(\tilde{f})(x))g(x)\right|dxdt\\
&\leq C\|g\|_{L^2(0,\infty)}\int_0^\infty \left[\left(\int_0^\infty \left|\frac{\partial}{\partial t}(W_t^{L_\alpha} (f)(x))\right|^2dx\right)^{1/2}+
\left(\int_0^\infty \left|\frac{\partial}{\partial t}(W_t^H(\tilde{f})(x))\right|^2dx\right)^{1/2}\right]dt\\
&\leq C\int_0^\infty \left[\left(\sum_{m=0}^\infty e^{-2\lambda _m^\alpha t} (\lambda _m^\alpha )^2|c_m^\alpha (f)|^2\right)^{1/2}+\left(\sum_{m=0}^\infty e^{-2\lambda _mt}(\lambda _m)^2|c_m(\tilde{f})|^2\right)^{1/2}\right]dt\\
&\leq C\int_0^\infty (e^{-\lambda _0^\alpha t}+e^{-\lambda _0t})dt<\infty .
\end{align*}
Hence, by interchanging the order of integration we get
\begin{align*}
\int_0^\infty \phi (t)\frac{d}{dt}\langle W_t^{L_\alpha} (f)-W_{t,+}(f),g\rangle_{L^2(0,\infty)}dt&\\
&\hspace{-6cm}= \Big\langle \int_0^\infty \phi
(t)\frac{\partial}{\partial t}(W_t^{L_\alpha} (f)(x)-W_{t,+}(f)(x))dt,g(x)\Big
\rangle _{L^2(0,\infty)}.
\end{align*}
Moreover, according to Proposition \ref{kernels}, it follows that
\begin{align*}
\int_0^\infty |\phi (t)|\int_0^\infty \left|\frac{\partial}{\partial t}(W_t^{L_\alpha} (x,y)-W_t^H(x,y))\right| |f(y)| dy&\\
&\hspace{-8cm}\leq
C\left(\int_{\frac{x}{2}}^{2x}\frac{1}{x}\left(1+\sqrt{\frac{x}{|y-x|}}\right)|f(y)|
dy+\frac{1}{x}\int_0^{\frac{x}{2}}|f(y)|dy+\int_{2x}^\infty
\frac{|f(y)|}{y}dy\right)\leq C,\quad x\in (0,\infty ),
\end{align*}
because $f\in C_c^\infty (0,\infty )$.

On the other hand, since $T_m^{L_\alpha} (f)\in L^2(0,\infty)$ and $T_m^H(\tilde{f})\in L^2(\mathbb{R})$, from \eqref{A1} we deduce that
\begin{align}\label{FF}
T_m^{L_\alpha} (f)(x)=&\int_0^\infty \phi (t)\Big(-\frac{\partial }{\partial t}\Big)[W_t^{L_\alpha} (f)(x)-W_{t,+}^H(f)(x)]dt\nonumber\\
&+\lim_{\varepsilon \rightarrow 0^+}\left(\Lambda (\varepsilon )f(x)+\int_{0,|x-y|>\varepsilon }^\infty  K_\phi^H(x,y) f(y) dy\right)\nonumber\\
=&\lim_{\varepsilon \rightarrow 0^+}\left(\Lambda (\varepsilon
)f(x)+\int_{0,|x-y|>\varepsilon }^\infty  K_\phi ^{L_\alpha}
(x,y)f(y) dy\right),\quad \mbox{a.e. }x\in (0,\infty ).
\end{align}
Thus \eqref{Tmexpression} and \eqref{Tmexpression2} are established for the Laguerre operator $L_\alpha$ and $f \in C_c^\infty(0,\infty)$.

Let us now  consider the maximal operator $T_m^{*,L_\alpha}$ defined by
$$
T_m^{*,L_\alpha}(f)(x)=\sup_{\varepsilon >0}\left|\Lambda
(\varepsilon)f(x)+\int_{0,|x-y|>\varepsilon }^\infty  K_\phi
^{L_\alpha} (x,y)f(y)dy\right|,\quad x\in (0,\infty ).
$$

Our objective now is to establish that $T_m^{*,L_\alpha}$ is a
bounded operator from $L^p(0,\infty)$ into itself, when
$1<p<\infty$ and from $L^1(0,\infty)$ into $L^{1,\infty
}(0,\infty)$. Since $\Lambda $ is a bounded function it is
sufficient to establish the result for
$$
\mathcal{T}_m^{*,L_\alpha}(f)(x)=\sup_{\varepsilon
>0}\left|\int_{0,|x-y|>\varepsilon }^\infty K_\phi ^{L_\alpha}
(x,y) f(y)dy\right|,\quad x\in (0,\infty ).
$$

By taking into account Proposition \ref{kernels} we can write
\begin{align*}
\mathcal{T}_m^{*,L_\alpha}(f)(x)\leq&\int_{(0,\infty)\setminus(\frac{x}{2},2x)}|K_\phi ^{L_\alpha} (x,y)||f(y)|dy+\int_{\frac{x}{2}}^{2x}|K_\phi ^{L_\alpha} (x,y)-K_\phi ^H(x,y)| |f(y)|dy\\
&+\sup_{\varepsilon >0}\left|\int_{\frac{x}{2},|x-y|>\varepsilon}^{2x} K_\phi^H(x,y) f(y)dy\right|\\
\leq &C\left(H_0^{\alpha +1/2}(|f|)(x)+H_\infty ^{\alpha
+1/2}(|f|)(x)+N(f)(x)+\mathcal{T}_{m ,{\rm
loc}}^{*,H}(f)(x)\right),\quad x\in (0,\infty ),
\end{align*}
where, for $\eta >-1$,
\begin{equation}\label{H0}
H_0^\eta (f)(x)=\frac{1}{x^{\eta +1}}\int_0^xy^\eta f(y)dy,\quad x\in (0,\infty ),
\end{equation}
\begin{equation}\label{Hinfty}
H_\infty ^\eta (f)(x)=x^\eta \int_x^\infty \frac{f(y)}{y^{\eta +1}}dy,\quad x\in (0,\infty ),
\end{equation}
\begin{equation}\label{N}
N(f)(x)=\int_{\frac{x}{2}}^{2x}\frac{1}{y}\left(1+\sqrt{\frac{x}{|x-y|}}\right)
|f(y)| dy,\quad x\in (0,\infty ),
\end{equation}
and
$$
\mathcal{T}_{m ,{\rm loc}}^{*,H}(f)(x)=\sup_{\varepsilon
>0}\left|\int_{\frac{x}{2},|x-y|>\varepsilon }^{2x}  K_\phi^H
(x,y)f(y)dy\right|,\quad x\in (0,\infty ).
$$

By using Jensen's inequality it can be seen that $N$ is a bounded
operator from $L^p(0,\infty )$ into itself, for every $1\leq p\leq
\infty$. Moreover, \cite[Lemmas 3.1 and 3.2]{ChH} show that the
Hardy type operators $H_0^{\alpha +1/2}$ and $H_\infty ^{\alpha
+1/2}$ are bounded from $L^p(0,\infty )$ into itself, when
$1<p<\infty $, and from $L^1(0,\infty )$ into $L^{1,\infty
}(0,\infty )$. Moreover, since $T_m^H$ is a Calder\'on-Zygmund
operator (see Proposition \ref{CZkernel}) the maximal operator
$\mathcal{T}_m^{*,H}$ defined by
$$
\mathcal{T}_m^{*,H}(g)(x)=\sup_{\varepsilon
>0}\left|\int_{|x-y|>\varepsilon }K_\phi^H
(x,y)g(y)dy\right|,\quad x\in \mathbb{R},
$$
is bounded from $L^p(\mathbb{R})$ into itself, for every
$1<p<\infty$, and from $L^1(\mathbb{R})$ into $L^{1,\infty
}(\mathbb{R})$. Then, according to Proposition \ref{kernels}, (a)
and (b), and using again \cite[Lemmas 3.1 and 3.2]{ChH}, we infer
that the operator $\mathcal{T}_{m,{\rm loc}}^{*,H}$ is bounded
from $L^p(0,\infty)$ into itself, for every $1<p<\infty$, and from
$L^1(0,\infty)$ into $L^{1,\infty }(0,\infty)$.

Thus we show the desired $L^p$-boundedness properties of the maximal operator $T_m^{*,L_\alpha}$.

By \eqref{FF} and since $C_c^\infty (0,\infty)$ is a dense subspace of $L^p(0,\infty)$, $1\leq p<\infty$, standard
arguments allow us to conclude that there exists the limit
$$
\mathbb{T}_m ^{L_\alpha} (f)(x)=\lim_{\varepsilon \rightarrow
0^+}\left(\Lambda (\varepsilon )f(x)+\int_{0,|x-y|>\varepsilon
}^\infty  K_\phi ^\alpha (x,y)f(y)dy\right),\quad \mbox{ a.e.
}x\in (0,\infty ),
$$
for $f\in L^p(0,\infty )$, $1\leq p<\infty$. Moreover,  the
operator $\mathbb{T}_m ^{L_\alpha}$ is bounded from
$L^p(0,\infty)$ into itself, for every $1<p<\infty$, and from
$L^1(0,\infty)$ into $L^{1,\infty }(0,\infty)$.

Then, $T_m^{L_\alpha}$ and $\mathbb{T}_m^{L_\alpha}$ are
$L^2$-bounded operators which coincides on $C_c^\infty (0,\infty
)$. Since this space is dense in $L^2(0,\infty )$, $T_m^{L_\alpha}
(f)=\mathbb{T}_m^{L_\alpha} (f)$, $f\in L^2(0,\infty)$.

Thus the proof is complete.

\subsection{Proof of Theorem \ref{UMD} for Laguerre operators}

Let $\alpha>-1/2$. The imaginary power $L_\alpha ^{i\gamma }$, $\gamma \in \mathbb{R}$, is the Laplace transform type multiplier for $L_\alpha$ associated to the function $m_\gamma (\lambda)=\lambda \int_0^\infty  e^{-\lambda t}\phi _\gamma (t)dt$, $\lambda >0$, where $\phi _\gamma (t)=(\Gamma (1-i\gamma ))^{-1}t^{-i\gamma }$, $t>0$.

Let $\B$ be a Banach space and $\gamma \in \mathbb{R}$.  We define $L^{i\gamma }_\alpha=T_{m_\gamma }^{L_\alpha} $ on $C_c^\infty (0,\infty )\otimes \B$ in the natural way. Also the imaginary power $H^{i\gamma}=T_{m_\gamma }^H$ of the Hermite operator $H$ is defined on $C_c^\infty (\mathbb{R})\otimes \B$.

We split the operator $T_{m_\gamma }^{L_\alpha}$ as follows:
\begin{equation}\label{split}
T_{m_\gamma }^{L_\alpha} (f)=T_{m_\gamma, {\rm
glob}}^{L_\alpha}(f)+D_{m_\gamma}^\alpha
(f)+T_{m_\gamma}^H(\tilde{f})-T_{m_\gamma, {\rm glob}}^H(f),
\end{equation}
where $f\in C_c^\infty(0,\infty )\otimes \B$ and $\tilde{f}\in C_c^\infty (\mathbb{R})\otimes \B$ is the extension of $f$ as given in Section 3 but with the natural Banach-valued sense, that is, if $f=\sum_{i=1}^nb_if_i$, $b_i\in \B$ and $f_i\in C_c^\infty (0,\infty )$, $i=1,...,n$, then $\tilde{f}=\sum_{i=1}^nb_i\tilde{f_i}$, where $\tilde{f_i}(x)=f_i(x)$, $x>0$, and $\tilde{f_i}(x)=0$, $x\leq 0$. The operators $T_{m_\gamma, {\rm glob}}^{L_\alpha}$, $D_{m_\gamma}^\alpha$ and $T_{m_\gamma, {\rm glob}}^H$ are defined on $C_c^\infty (0,\infty )\otimes \B$ in the following way:
$$
T_{m_\gamma, {\rm glob}}^{L_\alpha}(f)(x)=T_{m_\gamma }^{L_\alpha} (f\chi _{(0,\frac{x}{2})\cup (2x,\infty )})(x), \quad x\in (0,\infty ),$$
$$ D_{m_\gamma}^\alpha (f)(x)=\int_{\frac{x}{2}}^{2x}[K_{\phi _\gamma}^\alpha (x,y)-K_{\phi _\gamma }^H(x,y)]f(y)dy,\quad x\in (0,\infty ), $$
and
$$T_{m_\gamma, {\rm glob}}^H(f)(x)=T_{m_\gamma }^H(f\chi _{(0,\frac{x}{2})\cup (2x,\infty )})(x),\quad x\in (0,\infty ).
$$

These three operators can be extended to $L_\B^p(0,\infty )$, $1<p<\infty$, as bounded operators from $L_\B^p(0,\infty )$ into itself. Indeed, by using Proposition \ref{kernels} we can write, for every $f\in C_c^\infty (0,\infty )\otimes \B$,
$$
\|T_{m_\gamma, {\rm glob}}^{L_\alpha}(f)(x)\|_\B\leq C(H_0^{\alpha +1/2}(\|f\|_\B)(x)+H_\infty ^{\alpha +1/2}(\|f\|_\B)(x)), \quad x\in (0,\infty ),
$$
$$
\|D_{m_\gamma}^\alpha (f)(x)\|_\B\leq CN(\|f\|_\B)(x), \quad x\in (0,\infty ),
$$
and
$$
\|T_{m_\gamma, {\rm glob}}^{H}(f)(x)\|_\B\leq C(H_0^0(\|f\|_\B)(x)+H_\infty ^0(\|f\|_\B)(x)), \quad x\in (0,\infty ),
$$
where $H_0^\eta$, $H_\infty ^\eta$, $\eta =\alpha +1/2$ or $\eta
=0$, and $N$ are the operators given in \eqref{H0}, \eqref{Hinfty}
and \eqref{N}, respectively.

Then, by using Jensen's inequality and \cite[Lemmas 3.1 and
3.2]{ChH} we obtain that the above operators can be extended
boundedly to $L_\B^p(0,\infty)$, for every $1<p<\infty$.

 Theorem \ref{UMD} for the Hermite operators and \eqref{split} allow us to show that if $\B$ is UMD then, $T_{m_\gamma }^{L_\alpha} $ can be extended to $L_\B^p(0,\infty)$, $1<p<\infty $, as a bounded operator from $L_\B^p(0,\infty)$ into itself.

Assume now that, for some $1<p<\infty$, the operator $T_{m_\gamma
}^{L_\alpha} $ is extended boundedly to $L_\B^p(0,\infty)$. Then,
according to \eqref{split}, there exists $C>0$ such that
$$
\|T_{m_\gamma }^H(\tilde{f})\|_{L_\B^p(0,\infty)}\leq
C\|f\|_{L_\B^p(0,\infty )},\quad f\in L^2( 0,\infty)\otimes \B.
$$

Consider now $f\in L^2(\mathbb{R})\otimes \B$ and decompose $f$ as $f=f_1+f_2$, where $f_1=f\chi _{(0,\infty )}$. Thus, we have that
$$
\|T_{m_\gamma}^H(f)\|_{L_\B^p(\mathbb{R})}\leq
\sum_{i=1}^2\left(\|\chi _{(0,\infty )}T_{m_\gamma
}^H(f_i)\|_{L_\B^p(\mathbb{R})}+\|\chi _{(-\infty ,0]}T_{m_\gamma
}^H(f_i)\|_{L_\B^p(\mathbb{R})}\right).
$$

It is clear that
\begin{equation}\label{f1+}
\|\chi _{(0,\infty )}T_{m_\gamma
}^H(f_1)\|_{L_\B^p(\mathbb{R})}=\|T_{m_\gamma
}^H(\tilde{f}_{|(0,\infty)})\|_{L_\B^p(0,\infty)}\leq
C\|\tilde{f}_{|(0,\infty)}\|_{L_\B^p(0,\infty)}\leq
C\|f\|_{L_\B^p(\mathbb{R})}.
\end{equation}

Also, we have that
$$
T_{m_\gamma }^H(f_1)(-x)=\int_0^\infty K_{\phi
_\gamma}^H(-x,y)f(y)dy,\quad \mbox{ a.e. }x\in (0,\infty ).
$$
Then, Proposition \ref{kernels}(b) leads to
$$
\|T_{m_\gamma }^H(f_1)(-x)\|_\B\leq C\int_0^\infty \frac{\|f(y)\|_\B}{y+x}dy\leq C(H_0^0(\|f\|_\B)(x)+H_\infty ^0(\|f\|_\B)(x)),\quad \mbox{ a.e. }x\in (0,\infty ),
$$
which allows us to obtain, by using \cite[Lemmas 3.1 and 3.2]{ChH}, that
\begin{equation}\label{f1-}
\|\chi _{(-\infty ,0]}T_{m_\gamma
}^H(f_1)\|_{L_\B^p(\mathbb{R})}=\|T_{m_\gamma
}^H(f_1)(-\cdot)\|_{L_\B^p((0,\infty), dx)}\leq
C\|f\|_{L_\B^p(0,\infty)}\leq C\|f\|_{L_\B^p(\mathbb{R})}.
\end{equation}
On the other hand, since $K_{\phi _\gamma }^H(-x,-y)=K_{\phi _\gamma }^H(x,y)$, $x,y\in \mathbb{R}$, $x\not=y$, from Theorem 1.1 for the Hermite operator we deduce that
$$
T_{m_\gamma }^H(f_2)(-x)=\lim_{\varepsilon \rightarrow
0^+}\left(\Lambda (\varepsilon )f(-x) +\int_{0,|y-x|>\varepsilon
}^\infty K_{\phi _\gamma }^H(x,y)f(-y)dy\right)=T_{m_\gamma
}^H(\tilde{g})(x),\quad \mbox{ a.e. }x\in (0,\infty ),
$$
for a certain $\Lambda\in L^\infty(0,\infty)$, where $g(y)=f(-y)$, $y\in (0,\infty )$.
Thus,
\begin{equation}\label{f2-}
\|\chi _{(-\infty ,0]}T_{m_\gamma
}^H(f_2)\|_{L_\B^p(\mathbb{R})}=\|T_{m_\gamma
}^H(f_2)(-\cdot)\|_{L_\B^p(0,\infty)}=\|T_{m_\gamma
}^H(\tilde{g})\|_{L_\B^p(0,\infty)}\leq C\|f\|_{L_\B^p(\mathbb{R})}.
\end{equation}

Finally, since $g(y)=f(-y)$, $y\in (0,\infty )$, we can see that
$$
T_{m_\gamma }^H(f_2)(x)=\int_0^\infty K_{\phi _\gamma
}^H(x,-y)g(y)dy,\quad \mbox{ a.e. }x\in (0,\infty ),
$$
and, again by Proposition \ref{kernels}(b), that
$$
\|T_{m_\gamma }^H(f_2)(x)\|_\B\leq C(H_0^0(\|g\|_\B)(x)+H_\infty ^0(\|g\|_\B)(x)),\quad \mbox{ a.e. }x\in (0,\infty ).
$$
We obtain that
\begin{eqnarray}\label{f2+}
\|\chi _{(0,\infty )}T_{m_\gamma }^H(f_2)\|_{L_\B^p(\mathbb{R})}&=&\|T_{m_\gamma }^H(f_2)\|_{L_\B^p(0,\infty)}\nonumber\\&\leq &C(\|H_0^0(\|g\|_\B)\|_{L^p(0,\infty)}+\|H_\infty ^0(\|g\|_\B)\|_{L^p(0,\infty)})\nonumber\\
&\leq &C\|g\|_{L_\B^p(0,\infty )}\leq C\|f\|_{L_\B^p(\mathbb{R})}.
\end{eqnarray}

Estimations \eqref{f1+}-\eqref{f2+} lead to
$$
\|T_{m_\gamma}^H(f)\|_{L_\B^p(\mathbb{R})}\leq
C\|f\|_{L_\B^p(\mathbb{R})}.
$$
The proof finishes by using Theorem \ref{UMD} for the Hermite operator.



\end{document}